\documentclass[12pt]{amsart}
\usepackage[a4paper]{geometry}
\usepackage{inputenc}
\usepackage{amsmath}
\usepackage{amsfonts}
\usepackage{amssymb}
\usepackage{amsthm}
\usepackage{verbatim}
\usepackage{MnSymbol}

\theoremstyle{definition}

\newtheorem*{rem}{Remark}
\newtheorem{cor}{Corollary}

\theoremstyle{plain}
\newtheorem{thm}{Theorem}
\newtheorem{lem}[thm]{Lemma}
\newtheorem{prop}[thm]{Proposition}

\DeclareMathOperator{\tr}{tr}

\begin{document}

\title[$\mathcal{A}$-manifold over $\mathcal{A}$-manifold]{$\mathcal{A}$-manifolds on a principal torus bundle over an $\mathcal{A}$-manifold base}
\author[G. Zborowski]{Grzegorz Zborowski}
\address{Cracow University of Technology\\ Warszawska 24\\ 31-155 Krak\'ow, Poland}
\address{University of Maria Curie-Sk\l odowska\\ Pl. Marii Curie-Sk\l odowskiej 5, 20-035 Lublin, Poland}
\email{gzborowski@pk.edu.pl}

\begin{abstract}
We construct new examples of manifolds with cyclic-parallel Ricci tensor, so called $\mathcal{A}$-manifolds, on a $r$-torus bundle over a product of almost Hodge $\mathrm{A}$-manifolds.
\end{abstract}

\thanks{The author would like to thank prof. W. Jelonek.}

\maketitle

\section{Introduction}
One of the most extensively studied objects in mathematics and physics are Einstein manifolds (see for example \cite{besse}), i.e. manifolds whose Ricci tensor is a constant multiple of the metric tensor. In his work \cite{gray} A. Gray defined a condition which generalize the concept of an Einsten manifold. This condition states that the Ricci tensor $\mathrm{Ric}$ of the Riemannian manifold $(M,g)$ is cyclic parallel, i.e.
\begin{displaymath}
\nabla_X \mathrm{Ric}(Y,Z)+\nabla_Y \mathrm{Ric}(Z,X)+\nabla_Z \mathrm{Ric}(X,Y)=0,
\end{displaymath}
where $\nabla$ denotes the Levi-Civita connection of the metric $g$ and $X,Y,Z$ are arbitrary vector fields on $M$. A Riemannian manifold satisfying this condition is called an $\mathcal{A}$-manifold. It is obvious that if the Ricci tensor of $(M,g)$ is parallel, then it satisfies the above condition. On the other hand if $\mathrm{Ric}$ is cyclic-parallel, but not parallel then we call $(M,g)$ a strict $\mathcal{A}$-manifold. A. Gray gave in \cite{gray} first example of such strict $\mathcal{A}$-manifold, which was the sphere $S^3$ with appropriately defined homogeneous metric. A first example of a non-homogeneous $\mathcal{A}$-manifold was given in \cite{j1}. This example is a $S^1$-bundle over some K\"ahler-Einstein manifold. This result was generalized in \cite{j3} to K-contact manifolds. Namely, over every almost Hodge $\mathcal{A}$-manifold with $J$-invariant Ricci tensor we can construct a Riemannian metric such that the total space of the bundle is an $\mathcal{A}$-manifold. In the present paper we take a next step in the generalization process and we prove that there exist an $\mathcal{A}$-manifold structure on every $r$-torus bundle over product of almost Hodge $\mathcal{A}$-manifolds. Our result and that of Jelonek are based on the existence of almost Hodge $\mathcal{A}$-manifolds, which was proven in \cite{j2}.

\section{Conformal Killing tensors}

Let $(M,g)$ be any Riemannian manifold. We call a symmetric tensor field of type $(0,2)$ on $M$ a \textit{conformal Killing tensor} field iff there exists a $1$-form $P$ such that for any $X\in\Gamma(TM)$
\begin{equation}\label{Awar}
\nabla_XK(X,X) = P(X)g(X,X),
\end{equation}
where $\nabla$ is the Levi-Civita connection of $g$. The above condition is clearly equivalent to the following
\begin{equation}\label{cycAwar}
\mathcal{C}_{X,Y,Z}\nabla_X K(Y,Z) = \mathcal{C}_{X,Y,Z}P(X)g(Y,Z)
\end{equation}
for all $X,Y,Z\in\Gamma(TM)$ where $\mathcal{C}_{X,Y,Z}$ denotes the cyclic sum over $X,Y,Z$. It is easy to prove that the $1$-form $P$ is given by
\begin{equation*}
P(X) = \frac{1}{n+2}\left(2\mathrm{div} K(X)+d\tr K(X)\right),
\end{equation*}
where $X\in\Gamma(TM)$ and $\mathrm{div} S$ and $\tr S$ are the divergence and trace of the tensor field $S$ with respect to $g$.

If the $1$-form $P$ vanishes, then we call $K$ a \textit{Killing tensor}. Of particular interest in this work is a situation when the Ricci tensor of the metric $g$ is a Killing tensor. We call such a manifold an $\mathcal{A}$\textit{-manifold}. In the more general situation, when the Ricci tensor is a conformal tensor we call $(M,g)$ a $\mathcal{AC}^{\perp}$\textit{-manifold}.

We will use a following easy property of conformal Killing tensors.
\begin{prop}\label{killprod}
Suppose that $(M,g)$ is a Riemannian product of $(M_i,g_i)$, $i=1,2$. Moreover, let $K_i$ be conformal tensors on $(M_i,g_i)$. Then $K=K_1+K_2$ is a conformal tensor for $(M,g)$.
\end{prop}

A \textit{conformal Killing form} or a \textit{twistor form} is a differential $p$-form $\varphi$ on $(M,g)$ satisfying the following equation
\begin{equation}\label{confK}
\nabla_X\varphi = \frac{1}{p+1}X\righthalfcup d\varphi -\frac{1}{n-p+1}X\wedge\delta\varphi.
\end{equation}
An extensive description of conformal Killing forms can be found in a series of articles by Semmelmann and Moroianu (\cite{sem},\cite{s-m}). 

It is known that if $\varphi$ is a co-closed conformal Killing form (also called a \textit{Killing form}) then the $(0,2)$-tensor field $K_{\varphi}$ defined by
\begin{equation*}
K_{\varphi}(X,Y) = g(X\righthalfcup\varphi,Y\righthalfcup\varphi )
\end{equation*}
is a Killing tensor.

We can prove even more.
\begin{thm}\label{confKill}
Let $\varphi$ and $\psi$ be conformal Killing $p$-forms. Then the tensor field $K_{\varphi,\psi}$ defined by
\begin{equation*}
K_{\varphi,\psi} = g(X\righthalfcup\varphi,Y\righthalfcup\psi) + g(Y\righthalfcup\varphi,X\righthalfcup\psi)
\end{equation*}
is a conformal Killing tensor field.
\end{thm}
\begin{proof}
Let $X$ be any vector field and $\varphi$, $\psi$ conformal Killing $p$-forms. We will check that $K_{\varphi,\psi}$ as defined above satisfies \eqref{Awar}.
\begin{gather*}
\nabla_XK_{\varphi,\psi}(X,X) = 2X\left(g(X\righthalfcup\varphi,X\righthalfcup\psi)\right)-2g(\nabla_XX\righthalfcup\varphi,X\righthalfcup\psi)\\
-2g(\nabla_XX\righthalfcup\psi,X\righthalfcup\varphi)\\
=2g(\nabla_X(X\righthalfcup\varphi),X\righthalfcup\psi)+2g(X\righthalfcup\varphi,\nabla_X(X\righthalfcup\psi))\\
-2g(\nabla_XX\righthalfcup\varphi,X\righthalfcup\psi)-2g(\nabla_XX\righthalfcup\psi,X\righthalfcup\varphi)\\
=2g(X\righthalfcup\nabla_X\varphi,X\righthalfcup\psi)+2g(X\righthalfcup\varphi,X\righthalfcup\nabla_X\psi).
\end{gather*}
From the fact that $\varphi$ satisfies \eqref{confK} we have
\begin{gather*}
g(X\righthalfcup\nabla_X\varphi,X\righthalfcup\psi) = \frac{1}{p+1}g(X\righthalfcup(X\righthalfcup d\varphi),X\righthalfcup\psi)\\
-\frac{1}{n-p+1}g(X\righthalfcup(X\wedge\delta\varphi),X\righthalfcup\psi)\\
=-\frac{1}{n-p+1}\left(g(X,X)g(\delta\varphi,X\righthalfcup\psi)-g(X\wedge(X\righthalfcup\delta\varphi),X\righthalfcup\psi)\right)\\
=-\frac{1}{n-p+1}g(X,X)g(\delta\varphi,X\righthalfcup\psi).
\end{gather*}
The same is valid for $\psi$ with
\begin{equation*}
g(X\righthalfcup\nabla_X\psi,X\righthalfcup\varphi) = -\frac{1}{n-p+1}g(X,X)g(\delta\psi,X\righthalfcup\varphi).
\end{equation*}
Hence we have
\begin{equation*}
\nabla_XK_{\varphi,\psi}(X,X) = -\frac{2}{n-p+1}g(X,X)\left(g(\delta\varphi,X\righthalfcup\psi)+g(\delta\psi,X\righthalfcup\varphi)\right)
\end{equation*}
\end{proof}

\section{Torus bundles}

Let $(M,h)$ be a Riemannian manifold and suppose that $\beta_i$ are closed $2$-forms on $M$ for $i=1,\ldots, r$ such that their cohomology classes $[\beta_i]$ are integral. In \cite{koba} it was proven that to each such cohomology class there corresponds a principal circle bundle $p_i : P_i\rightarrow M$ with a connection form $\theta_i$ such that 
\begin{equation}\label{cxform}
d\theta_i=2\pi p_i^*\beta_i.
\end{equation}
Taking the Whitney sum of bundles $(p_i,P_i,M)$ we obtain a principal $r$-torus bundle $p: P\rightarrow M$ classified by cohomology classes of $\beta_i$, $i=1,\ldots, r$. The connection form $\theta$ is a vector valued $1$-form with coefficients $\theta_i$, where $\theta_i$ are as before. For each connection form $\theta_i$ we define a vector field $\xi^i$ by $\theta_i(\xi^i)=1$. This vector field is just the fundamental vector field for $\theta_i$ corresponding to $1$ in the Lie algebra of $i$-th $S^1$-factor of the bundle $(p,P,M)$.

It is easy to check that the tensor field $g$ given by
\begin{equation}\label{metricg}
g = \sum_{i,j=1}^rb_{ij}\theta_i\otimes\theta_j + p^*h
\end{equation}
is a Riemannian metric on $P$ if $[b_{ij}]_{i,j=1}^r$ is some symmetric, positive definite $r\times r$ matrix with real coefficients. This Riemannian metric makes the projection $p: (P,g)\rightarrow (M,h)$ a Riemannian submersion (see \cite{oneill}).

\begin{lem}
Each vector field $\xi^i$ for $i=1,\ldots, r$ is Killing with respect to the metric $g$. Moreover, define a tensor field $T_i$ of type $(1,1)$ by $T_iX=\nabla_X\xi^i$ for $X\in\Gamma(TP)$, where $\nabla$ is the Levi-Civita connection of $g$. Then we have
\begin{equation*}
T_i\xi^j=0,\quad L_{\xi^i}T_j=0,
\end{equation*}
for $i\neq j$.
\end{lem}
\begin{proof}
To prove that $\xi^s$ is a Killing vector field for $s=1,\ldots,r$ observe that
\begin{equation*}
L_{\xi^s}g = \sum_{i,j=1}^rb_{ij}\left(\left(L_{\xi^s}\theta_i\right)\otimes\theta_j + \theta_i\otimes\left(L_{\xi^s}\theta_j\right)\right).
\end{equation*}
Hence we only have to check that $L_{\xi^s}\theta_i=0$ for any $i,s=1,\ldots,r$. Using Cartan's magic formula for Lie derivative we have
\begin{equation*}
L_{\xi^s}\theta_i = d\left(\theta_i(\xi^s)\right) + \xi^s\righthalfcup d\theta_i
\end{equation*} 
and it is immediate that the first term is zero, since $\theta_i(\xi^s)=\delta_i^s$, where $\delta_i^s$ is the Kronecker delta. For the second term we have
\begin{equation}
d\theta_i(\xi^s,X) = \xi^s\left(\theta_i(X)\right) - X\left(\theta_i(\xi^s)\right)-\theta_i\left([\xi^s,X]\right),
\end{equation}
where $X$ is arbitrary. We will consider two cases, namely when $X$ is a horizontal or vertical vector field. In both cases the first two components vanish, hence we only have to look at the third. In the first case we notice that $[\xi^s,X]$ is a horizontal vector field, since $\xi^s$ is a fundamental vector field on $P$. This gives us the vanishing of $\xi^s\righthalfcup d\theta_i$ on horizontal vector fields. When $X$ is vertical we can take it to be just $\xi^k$ and we immediately see that $[\xi^s,\xi^k]=0$ since the fields $\xi^j$ come from the action of a torus on $P$.

For the second part of the lemma observe that $g(\xi^i,\xi^j)$ is constant. For any vector field $X$ this gives us
\begin{equation*}
0= Xg(\xi^i,\xi^j) = g(\nabla_X\xi^i,\xi^j)+g(\xi^i,\nabla_X\xi^j) = -g(X,\nabla_{\xi^j}\xi^i)-g(\nabla_{\xi^i}\xi^j,X). 
\end{equation*} 
Now, since $[\xi^i,\xi^j]=0$ we have $\nabla_{\xi^i}\xi^j=\nabla_{\xi^j}\xi^i$ which proves that $T_i\xi^j=0$.

Recall that for any Killing vector field we have
\begin{equation*}
L_{\xi}\nabla_XY = \nabla_{L_{\xi}X}Y + \nabla_X(L_{\xi}Y),
\end{equation*}
where $X$ and $Y$ are arbitrary vector fields. In our situation we have
\begin{equation*}
(L_{\xi^i}T_j)X = L_{\xi^i}(T_jX) - T_j(L_{\xi^i}X) = \nabla_{[\xi^i,X]}\xi^j + \nabla_X[\xi^i,\xi^j] - \nabla_{[\xi^i,X]}\xi^j = 0,
\end{equation*}
which ends the proof.
\end{proof}

Hence tensor fields $T_i$ are horizontal, i.e. for each $i$ there exists a tensor field $\tilde{T}_i$ on $M$ such that $p_*\circ T_i=\tilde{T}_i\circ p_*$.

We now compute the O'Neill tensors (\cite{oneill}) of the Riemannian submersion $p:P\rightarrow M$. 
\begin{prop}
The O'Neill tensor $T$ is zero. The O'Neill tensor $A$ is given by
\begin{equation}\label{tensorA}
A_EF = \sum_{i,j=1}^rb^{ij}\left(g(E,T_iF)\xi^j + g(\xi^i,F)T_jE\right),
\end{equation}
where $b^{ij}$ are the coefficients of the inverse matrix of $[b_{ij}]_{i,j=1}^r$ and $E,F\in\Gamma(TP)$.
\end{prop}

Observe that from the fact that $\theta_i(\xi^i) = 1$ for $E\in\Gamma(TP)$ we get that
\begin{equation*}
g(\xi^i,E) = \sum_{j=1}^rb_{ij}\theta_j(E)
\end{equation*}
hence
\begin{equation*}
\theta_j(E) = \sum_{i=1}^rb^{ji}g(\xi^i,E).
\end{equation*}
Taking the exterior differential we get
\begin{equation}\label{extheta}
d\theta_j(E,F) = 2\sum_{i=1}^r b^{ji}g(T_iE,F),
\end{equation}
where $E,F\in\Gamma(TP)$.

Using formulae from \cite{besse} Chapter 9 and the fact that the fibre of the Riemannian submersion $(p,P,M)$ is totally geodesic and flat, we see that the Ricci tensor on the total space of Riemannian submersion is given by
\begin{align}
\label{ricV}\mathrm{Ric}(U,V) &= \sum_{i=1}^mg(A_{E_i}U,A_{E_i}V),\\
\label{ricVH}\mathrm{Ric}(X,U) &= -\sum_{i=1}^m g\left((\nabla_{E_i}A)_{E_i}X,U\right),\\
\label{ricH}\mathrm{Ric}(X,Y) &= \mathrm{Ric}_M(X,Y) - 2\sum_{i=1}^m g(A_XE_i,A_YE_i).
\end{align}
Here $E_i$ is an element of the orthonormal basis of the horizontal distribution $\mathcal{H}$, $\mathrm{Ric}_M$ is a lift of the Ricci tensor of the base $(M,h)$, $X,Y$ are horizontal vector fields and $U,V$ any vertical vector fields. Using the formula \eqref{tensorA} for the O'Neill tensor $A$ we can compute all components of the Ricci tensor $\mathrm{Ric}$. We obtain
\begin{align}
\label{ricV1}\mathrm{Ric}(U,V) &= \sum_{i=1}^mg\left(\sum_{s,t=1}^rb^{st}g(\xi^s,U)T_tE_i,\sum_{k,l=1}^rb^{kl}g(\xi^k,V)T_lE_i\right),\\
\label{ricH1}\mathrm{Ric}(X,Y) &= \mathrm{Ric}_M(X,Y) - \frac{1}{2}\sum_{s,t=1}^rb^{st}g(T_sX,T_tY).
\end{align}
As for the value of $\mathrm{Ric}(X,U)$ we compute the covariant derivative
\begin{gather*}
\left(\nabla_{E_i}A\right)_{E_i}X = \nabla_{E_i}\left(\sum_{s,t=1}^rb^{st}g\left(E_i,T_sX\right)\xi^t\right)-\sum_{s,t=1}^rb^{st}g\left(\nabla_{E_i}E_i,T_sX\right)\xi^t\\
-\sum_{s,t=1}^rb^{st}\left(g\left(E_i,T_s\nabla_{E_i}X\right)\xi^t+g\left(\xi^s,\nabla_{E_i}X\right)T_tE_i\right)\\
=\sum_{s,t=1}^rb^{st}g\left(E_i,\left(\nabla_{E_i}T_s\right)X\right)\xi^t,
\end{gather*}
where we used the fact that $g(\xi^s,\nabla_{E_i}X)=-g(T_sE_i,X)$ which follows from $A_X$ being anti-symmetric with respect to $g$ for any horizontal vector field $X$. Now since tensors $T_s$ are anti-symmetric with respect to $g$ so is $\nabla_XT_s$, hence
\begin{equation*}
\left(\nabla_{E_i}A\right)_{E_i}X = -\sum_{s,t=1}^rb^{st}g\left(\left(\nabla_{E_i}T_s\right)E_i,X\right)\xi^t = \sum_{t=1}^r\delta d\theta_t(X)\xi^t.
\end{equation*}
As a result we have
\begin{equation*}
\mathrm{Ric}(X,U) = \sum_{t=1}^r\delta d\theta_t(X) g(\xi^t,U).
\end{equation*}

\section{Torus bundle over a product of almost Hodge manifolds}

Recall that an almost complex manifold is a pair $(M,J)$ where $M$ is a some differential manifold and $J$ is an endomorphism of $TM$ such that $J^2=-\mathrm{id}_{TM}$. On such manifolds we can single out particular Riemannian metrics which we call compatible with the almost complex structure $J$. The compatibility condition for a metric $g$ is $g(JX,JY)=g(X,Y)$. We call such metrics almost Hermitian and the triple $(M,g,J)$ an almost Hermitian manifold. We can define a differential $2$-form $\omega$ by $\omega(X,Y) = g(JX,Y)$. We call $\omega$ a K\"ahler form of $(M,g,J)$. If the complex structure $J$ is integrable and the K\"ahler form is closed we call such complex manifold a K\"ahler manifold. Such manifolds are of no use in this work, since by a result of Sekigawa and Vanhecke (\cite{sv}) every K\"ahler $\mathcal{A}$-manifold has parallel Ricci tensor.

There are however manifolds very close to being K\"ahler which are more suitable for us. Let $(M,g,J)$ be an almost complex manifold with closed K\"ahler form. We call such manifolds \textit{almost K\"ahler manifolds}. If the almost complex structure $J$ is not integrable then $(M,g,J)$ is sometimes called a \textit{strictly almost K\"ahler manifold}. In addition to being closed the K\"ahler form of an almost K\"ahler manifold is also coclosed, hence harmonic with respect to $g$.

In \cite{j2} Jelonek constructed a strictly almost K\"ahler $\mathcal{A}$-manifold with non-parallel Ricci tensor. Moreover the K\"ahler form of such a manifold has a useful property. It is a constant multiple of some differential $2$-form that belongs to an integral cohomology class i.e. a differential form in $H^2(M;\mathbb{Z})$. An almost K\"ahler manifold whose K\"ahler form satisfies this condition is called an almost Hodge manifold.

Returning to our construction suppose that $(M_i,g_i,J_i)$, $i=1,\ldots, n$ are almost Hodge manifolds such that K\"ahler forms $\omega_i$ are constant multiples of $2$-forms $\alpha_i$ and their cohomology classes are integral, i.e. $[\alpha_i]\in H^2(M_i;\mathbb{Z})$. Denote by $(M,g,J)$ the product manifold with the product metric and product almost complex structure and let $pr_i$ be the projection on the $i$-th factor. From our earlier discussion we know that there exists a principal $r$-torus bundle classified by the forms $\beta_1,\ldots, \beta_r$ given by
\begin{equation*}
\beta_j = \sum_{i=1}^n a_{ji}pr_i^*\alpha_i,
\end{equation*}
where $[a_{ji}]$ is some $r\times n$ matrix with integer coefficients. By \eqref{cxform} the coefficients $\theta_j$ of the connection form of $(p,P,M)$ satisfy
\begin{equation*}
d\theta_j = 2\pi p^*\beta_j=2\pi \sum_{i=1}^n a_{ji}p^*\left(pr_i^*\alpha_i\right)
\end{equation*}
for every $j=1,\ldots, r$. Since $\alpha_i$'s and K\"ahler forms $\omega_i$ of $(M_i,g_i,J_i)$ are connected by $\omega_i=c_i\alpha_i$ for some constants $c_i$, $i=1,\ldots,n$ we have
\begin{equation}\label{cxform-Kf}
d\theta_j = 2\pi\sum_{i=1}^n \frac{a_{ji}}{c_i}\omega_i^*,
\end{equation}
where by $\omega_i^*$ we denote the $2$-form obtained from lifting $\omega_i$ to $P$. Comparing this with \eqref{extheta} we get a formula for each tensor field $\tilde{T}_i$
\begin{equation}\label{tensorT}
\tilde{T}_iX = \pi \sum_{j=1}^rb_{ij}\sum_{k=1}^n \frac{a_{jk}}{c_k}J_k^*X
\end{equation}
where $J_k^*$ is the almost complex structure tensor of $(M_k,g_k,J_k)$ lifted to the product manifold $M$.

We will now compute the Ricci tensor of $(P,g)$ using \eqref{ricV}-\eqref{ricH}, computations that follows those formulas and above observations. We begin with
\begin{gather}\label{ricV2}
\mathrm{Ric}(U,V) = \pi^2\sum_{i=1}^mh\left(\sum_{s=1}^rg(\xi^s,U)\sum_{k=1}^n\frac{a_{sk}}{c_k}J^*_kE_i,\sum_{l=1}^rg(\xi^l,U)\sum_{h=1}^n\frac{a_{lh}}{c_h}J^*_hE_i\right)\\
\nonumber=\pi^2\sum_{s,l=1}^rg(\xi^s,U)g(\xi^l,V)\sum_{i=1}^mh\left(\sum_{k=1}^n\frac{a_{sk}}{c_k}J^*_kE_i,\sum_{h=1}^n\frac{a_{lh}}{c_h}J^*_hE_i\right)\\
\nonumber=\pi^2\sum_{s,l=1}^rg(\xi^s,U)g(\xi^l,V)\sum_{i=1}^m\sum_{k=1}^ng_k\left(\frac{a_{sk}}{c_k}J_kE_i,\frac{a_{lk}}{c_k}J_kE_i\right).
\end{gather}
We used the fact that for $k\neq h$ images of $J_k$ and $J_h$ are orthogonal. It is easy to see that
\begin{equation*}
\sum_{i=1}^m\sum_{k=1}^ng_k\left(\frac{a_{sk}}{c_k}J_kE_i,\frac{a_{lk}}{c_k}J_kE_i\right)
\end{equation*}
are constants for each $s,l=1,\ldots,r$. Hence the Ricci tensor of $(P,g)$ on vertical vector fields is a symmetrized product of Killing vector fields.

Next, since the K\"ahler form of each almost Hodge manifold $(M_k,g_k,J_k)$ is co-closed we see from \eqref{cxform-Kf} that
\begin{equation}\label{ricVH2}
\mathrm{Ric}(X,U) = 0
\end{equation}
for any horizontal vector field $X$ and vertical vector field $U$.

The last component of the Ricci tensor of $(P,g)$ is the horizontal one. First observe that $\mathrm{Ric}_M$ is the Ricci tensor of the product metric $h=g_1+\ldots+g_n$ and Ricci tensors $\mathrm{Ric}_k$ are $J_k$-invariant Killing tensors. We have
\begin{thm}\label{liftKill}
Let $K_i$ be a Killing tensor on $(M_i,g_i,J_i)$ for $i=1,\ldots,n$. Then the lift $K^*$ of $K=K_1+\ldots+K_n$ to $P$ is a Killing tensor iff each $K_i$ is $J_i$-invariant.
\end{thm}
\begin{proof}
We need to check the cyclic sum condition \eqref{cycAwar} for different choices of vector fields. It is easy to see that if all three vector fields are vertical then each component of the cyclic sum vanishes, since $K^*$ is non-vanishing only on horizontal vector fields. If only two of the vector fields are vertical then again all components vanish, since $\nabla_{\xi^i}\xi^j = 0$. For three horizontal vector fields we again see that the cyclic sum vanish, since the covariant derivative of $K^*$ with respect to metric $g$ on $P$ is the same as that of $K$ with respect to the product metric $h$ on $M$. By Proposition \ref{killprod} $K$ is a Killing tensor for $(M,h)$. The remaining case is when only one vector field is vertical. Let us put $Z=\xi^i$ and $X,Y$ be basic horizontal vector fields. We compute
\begin{align*}
\nabla_{\xi^i}K^*(X,Y) &= -K^*(\nabla_{\xi^i}X,Y) - K^*(X,\nabla_{\xi^i}Y) = -K^*(A_X\xi^i,Y)-K^*(X,A_Y\xi^i) \\
&= -K^*(\nabla_X\xi^i,Y) - K^*(X,\nabla_Y\xi^i),
\end{align*}
where the before last equality is due to the fact that $X$ and $Y$ are basic (see \cite{oneill}) and the last one follows from the definition of the O'Neill tensor $A$. Next we have
\begin{equation*}
\nabla_XK^*(\xi^i,Y) = -K^*(\nabla_X\xi^i,Y).
\end{equation*} 
Summing up we have
\begin{align*}
\mathcal{C}_{\xi^i,X,Y}\nabla_{\xi^i}K^*(X,Y) &= -2\left(K^*(\nabla_X\xi^i,Y) + K^*(X,\nabla_Y\xi^i)\right)\\
&= -2\left(K(\tilde{T}_iX,Y)+K(X,\tilde{T}_iY)\right).
\end{align*}
Now we use the formula \eqref{tensorT} for the tensor $\tilde{T}_i$
\begin{equation*}
\mathcal{C}_{\xi^i,X,Y}\nabla_{\xi^i}K^*(X,Y) = -2\pi\sum_{j=1}^rb_{ij}\sum_{k=1}^n \frac{a_{jk}}{c_k}\left(K(J_k^*X,Y) + K(X,J_k^*Y)\right).
\end{equation*}
Since each $J_i$ projects vector fields on $TM_k$ we see from the definition of $K$ that
\begin{equation*}
K(J_k^*X,Y) + K(X,J_k^*Y) = K_k(J_kX,Y) + K_k(X,J_kY).
\end{equation*}
By $J_k$-invariance of $K_k$ for $k=1,\ldots, n$ we have completed the proof.
\end{proof}

\begin{rem}
It is worth noting, that we cannot lift in that way a conformal Killing tensor with non-vanishing $P$. In fact taking three vertical vector fields we see that $P$ vanishes on vertical distribution. On the other hand for two vertical vector fields $U,V$ and one horizontal vector field $X$ the left-hand side of \eqref{cycAwar} vanish and the right-hand side reads $P(X)g(U,V)$, hence $P$ has to vanish also on the horizontal distribution.
\end{rem}

\begin{cor}
An $r$-torus bundle with metric defined by \eqref{metricg} can not be an $\mathcal{AC}^{\perp}$\textit{-manifold}. Especially there are no $\mathcal{AC}^{\perp}$ structures on K-contact and Sasakian manifolds.
\end{cor}

Next we show that the second component of the horizontal part of the Ricci tensor \eqref{ricH1} is just a sum of lifts of metrics $g_k$, $k=1,\ldots,n$.
\begin{equation*}
\sum_{s,t=1}^rb^{st}g(T_sX,T_tY) = \pi^2\sum_{s,t=1}^rh\left(\sum_{j=1}^rb_{sj}\sum_{k=1}^n\frac{a_{jk}}{c_k}J_k^*X,\sum_{i=1}^rb_{ti}\sum_{l=1}^n\frac{a_{il}}{c_l}J_l^*Y\right).
\end{equation*}
Since $J_k$ and $J_l$ are orthogonal for different $k,l=1,\ldots,n$ we obtain
\begin{gather}\label{ricH2}
\sum_{s,t=1}^rb^{st}g(T_sX,T_tY) = \pi^2\sum_{s,t=1}^r\sum_{k=1}^nh\left(\sum_{j=1}^rb_{sj}\frac{a_{jk}}{c_k}J_k^*X,\sum_{i=1}^rb_{ti}\frac{a_{ik}}{c_k}J_k^*Y\right)\\
\nonumber=\pi^2\sum_{j,l=1}^rb_{jl}\sum_{k=1}^r\frac{a_{jk}a_{lk}}{c_k^2}g_k(X,Y).
\end{gather}
From the above Theorem we infer that, since a Riemannian metric is a Killing tensor and each $g_k$ is $J_k$-invariant, the tensor field $K(X,Y) = \sum_{s,t=1}^rb^{st}g(T_sX,T_tY)$ is a Killing tensor field.

Now we can prove the following theorem
\begin{thm}
Let $P$ be a $r$-torus bundle over a Riemannian product $(M,h)$ of almost Hodge $\mathcal{A}$-manifolds $(M_k,g_k,J_k)$, $k=1,\dots n$ with metric $g$ defined by \eqref{metricg}. Then $(P,g)$ is itself an $\mathcal{A}$-manifold.
\end{thm}
\begin{proof}
Since distributions $\mathcal{H}$ and $\mathcal{V}$ are orthogonal with respect to the Ricci tensor $\mathrm{Ric}$ of $(P,g)$ by \eqref{ricVH2} we can write it as 
\begin{gather*}
\mathrm{Ric}(E,F) = \pi^2\sum_{s,l=1}^rg(\xi^s,E)g(\xi^l,F)\sum_{i=1}^m\sum_{k=1}^ng_k\left(\frac{a_{sk}}{c_k}J_kE_i,\frac{a_{lk}}{c_k}J_kE_i\right)\\
 +\mathrm{Ric}_M(E,F) -\frac{1}{2}\pi^2\sum_{j,l=1}^rb_{jl}\sum_{k=1}^r\frac{a_{jk}a_{lk}}{c_k^2}g_k(E,F)
\end{gather*}
using \eqref{ricH2} and \eqref{ricV2}.The first component is a Killing tensor as a symmetrized product of Killing vector fields by Theorem \ref{confKill}. The second and third components are Killing tensors by Theorem \ref{liftKill}. Since a sum of Killing tensors with constant coefficients is again a Killing tensor we have proved the theorem.
\end{proof}

\begin{rem}
Observe that if at least one of the manifolds $(M_k,g_k)$ has non-parallel Ricci tensor, then the Ricci tensor $\mathrm{Ric}$ of $(P,g)$ is also non-parallel with respect to the metric $g$. Thus we have constructed a large number of strict $\mathcal{A}$-manifolds.
\end{rem}

\end{document}